\documentclass[12pt,a4paper]{amsart}
\usepackage{color,amsmath,amssymb,latexsym,amsfonts,graphicx}
\usepackage[parfill]{parskip}
\usepackage{url}
\usepackage{changes}
\usepackage[framemethod=TikZ]{mdframed}

\newtheorem{theorem}{Theorem}
\newtheorem{corollary}[theorem]{Corollary}
\newtheorem{example}[theorem]{Example}
\newtheorem{lemma}[theorem]{Lemma}
\newtheorem{proposition}[theorem]{Proposition}
\newtheorem{remark}[theorem]{Remark}

\def\bds{\begin{displaystyle}}
\def\eds{\end{displaystyle}}

\def\eqd{\,{\buildrel d \over =}\,}

\def\1{{\mathchoice {1\mskip-4mu\mathrm l}      
{1\mskip-4mu\mathrm l}
{1\mskip-4.5mu\mathrm l} {1\mskip-5mu\mathrm l}}}
\newcommand{\floor}[1]{\left\lfloor #1 \right\rfloor}

\mdfdefinestyle{MyFrame}{%
    linecolor=gray!15!white,
    backgroundcolor=gray!40!white}
\mdfdefinestyle{MyFrame1}{%
    skipbelow=1.5ex plus2pt minus2pt,%
    skipabove=1.5ex plus 2pt minus2pt,%
    linecolor=gray!40!white,%
    backgroundcolor=gray!9!white}
\mdfdefinestyle{MyFrame2}{%
    linecolor=gray!40!white,
    backgroundcolor=gray!0!white}

\newcommand{\sgn}{\hbox{\rm sgn}}
\newcommand{\sgnn}{\hbox{\sc sgn}}

\begin{document}

\title{Limit theorems and ergodicity for general bootstrap random walks}
\author{A. Collevecchio, K. Hamza, M. Shi, R.J. Williams}

\date{}

\maketitle

\begin{abstract}
Given the increments of a simple symmetric random walk $(X_n)_{n\ge0}$, we characterize all possible ways of recycling these increments into a simple symmetric random walk $(Y_n)_{n\ge0}$ adapted to the filtration of $(X_n)_{n\ge0}$. We study the long term behavior of a suitably normalized two-dimensional process $((X_n,Y_n))_{n\ge0}$. In particular, we provide necessary and sufficient conditions for the process to converge to a two-dimensional Brownian  motion (possibly degenerate).  We also discuss cases in which the limit is not Gaussian. Finally, we provide a simple necessary and sufficient condition for the ergodicity of the recycling transformation, thus generalizing results from Dubins and Smorodinsky (1992) and Fujita (2008), and solving the discrete version of the open problem of the ergodicity of the general L\'evy transformation (see Mansuy and Yor, 2006).
\end{abstract}

\section{Introduction}\label{intro}
Simple symmetric random walks, stochastic integrals and measure preserving transformations are ubiquitous to the theory of probability as well as to a wide range of applications. This paper investigates the properties of the most general discrete-time setting that combines all three; that is the most general measure-preserving stochastic integral of a simple symmetric random walk.

More specifically, consider a one-dimensional simple symmetric random walk $(X_n)_n$, with $X_0 = 0$. Let $\xi_n = X_n- X_{n-1}$ be the independent increments of the random walk, with common distribution
$\mathbb{P}(\xi_1=-1) = \mathbb{P}(\xi_1=+1) = 1/2$. We study the measure-preserving, non-anticipative (i.e. adapted) bootstrapping (i.e. recycling) of $(\xi_n)_{n\ge1}$ and obtain a complete description of all functions $\phi_n$ such that the sequence $\eta_n=\phi_n(\xi_1,\ldots,\xi_n)$ replicates the law of the original sequence: $(\eta_n)_{n\ge1}\eqd(\xi_n)_{n\ge1}$.
Such a sequence defines a new simple symmetric random walk
$$Y_n=\sum_{k=1}^n\eta_k,\ n\geq1,\mbox{ and }Y_0=0.$$
Equipped with such a representation, we study the limiting behaviour of a suitably normalized pair $(X_n,Y_n)$ as well as the ergodicity of the recycling transformation.

Seen from the point of view of the random walks themselves, rather than their increments, we aim to study the long term behaviour of the two-dimensional General Bootstrap Random Walk (GBRW) $(X_n,Y_n)$, where $(Y_n)_{n\ge0}$ is also a simple symmetric random walk adapted to the natural filtration of $(X_n)_{n\ge0}$. As such, by the martingale representation property (see for example 15.1 of \cite{Williams91}), $$\eta_n=Y_n-Y_{n-1}=H_{n-1}(X_n-X_{n-1})=H_{n-1}\xi_n,$$
for some adapted process $(H_n)_{n\ge0}$. This combined with the fact that $(Y_n)_{n\ge0}$ is a simple random walk, or equivalently that $\eta_n\in\{-1,1\}$, enable the characterisation and parametrisation of the processes $(H_n)_{n\ge0}$ -- see Section \ref{Section:NaB}. In Section \ref{Section:Gauss}, we give necessary and sufficient conditions for a suitably normalized $(X_n,Y_n)$ to converge weakly to a two-dimensional Brownian motion (possibly degenerate). Section \ref{Section:NonGauss} explores the case of symmetric functions, such as $\eta_n=\sgn(X_{n-1})\xi_n$, and shows that a non-Gaussian limit results from any such symmetric bootstrapping.

Besides describing the asymptotic behaviour of the pair $(X_n,Y_n)$, the paper examines the ergodicity of the recycling transformation and provides a simple necessary and sufficient condition for the latter to hold. This is detailed in Section \ref{Section:Ergodic}.

\subsection*{Literature review}
\underline{The case $\eta_n = \prod_{k=1}^n\xi_k$} has been the object of a number of investigations in a variety of contexts. It was referred to as the Bootstrap Random Walk (BRW) and described in great details in \cite{CHS2016} (see also \cite{CHL2019}).  In mathematical finance, it was used to discuss the continuity of utility maximization under weak convergence (see \cite{Dolinsky2020}). {Within the context of noise sensitivity, \cite{Prigent2020} compares the effects on the sequences $(X_n)_{n\ge0}$ and $(Y_n)_{n\ge0}$ of ``Poisson switches'' in the sequence $(\xi_n)_{n\ge1}$. Owing to the fact that every switch in the sequence $(\xi_n)_{n\ge1}$ results in multiple concurrent switches in the sequence $(\eta_n)_{n\ge1}$, it is shown in \cite{Prigent2020} that the noise sensitivity of $(Y_n)_{n\ge0}$ is grater than that of $(X_n)_{n\ge0}$.}

GBRW bear some relationship to the celebrated elephant random walk introduced by Sch\"utz and Trimper \cite{Schutz2004}. This  is a model for long memory within a random walk setting. In it, the cumulative effect of BRW $\eta_n=\eta_{n-1}\xi_n$ is replaced with $\eta_n=\eta^*\xi_n$, where $\eta^*$ is selected at random from $\{\eta_1,\ldots,\eta_{n-1}\}$. The process is shown in \cite{Schutz2004} to undergo a phase transition at the critical value $p_c=1/2$, from a weakly localized regime to an escape regime. We refer the reader to \cite{Gut}, \cite{Masato} and \cite{Ber1}, and the reference lists within.

Finally, \cite{Volk1} and \cite{Volk2} investigate a time-dependent biased bootstrap random walk and show that as the ``turning'', that is $\mathbb{P}(\xi_n=-1)$, gets slower, the Central Limit Theorem and then the Law of Large Numbers break down.

\underline{The case  $\eta_n=\sgn(X_{n-1})\xi_n$} is the discrete version of, and as such is closely related to, the celebrated L\'evy transformation $$B_\bullet \rightarrow\int_0^\bullet \sgn(B_s)dB_s,$$
where $B$ is a Brownian motion.
The ergodicity of the latter is to this date an open problem (see for example \cite{Prokaj}). The ergodicity of more general, L\'evy-type transformations, $B_\bullet\rightarrow\int_0^\bullet H_sdB_s$, for a predictable process $H_t \in \{-1, 1\}$, is also of significant interest. \cite{MansYor} explicitly asks to find a characterisation of the predictable processes $H$ for which ergodicity holds true. \cite{MansYor} adds that `this seems to be an extremely difficult question, the solution to which has escaped so far both Brownian motion and ergodic theory experts'. We solve, completely, this question in the discrete setting. A previous work by Fujita \cite{Fujita} showed that the discrete L\'evy transformation is not ergodic. We refer the interested reader to \cite{DubinsSmorodinsky} for a modified version for which ergodicity does hold. A consequence of our results is that if we pick a recycling rule uniformly at random, then it is almost surely non-ergodic.

\section{Non-anticipative bootstrapping -- Two representations}\label{Section:NaB}

Let $\xi_1,\xi_2,\ldots$ be independent and identically distributed random variables with
$$\mathbb{P}(\xi_1=-1) = \mathbb{P}(\xi_1=+1) = \frac12.$$
Given a sequence of functions $\phi_n:\{-1,+1\}^n\longrightarrow\{-1,+1\}$, we define
$$\eta_n=\phi_n(\xi_1,\ldots,\xi_n).$$
As mentioned above, a direct consequence of the martingale representation property for the simple symmetric random walk $(X_n)_{n\ge0}$ yields a first representation of $\eta_n$ in terms of the sequence $(\xi_n)_{n\ge1}$.
\begin{mdframed}
[style=MyFrame1]
\begin{proposition}\label{genphiprop}
$(\eta_n)_{n\ge1}\stackrel{d}{=}(\xi_n)_{n\ge1}$ if and only if $\phi_n(u_1,\dots,u_n)$ is of the following form:
\begin{equation}\label{genphi}
\phi_n(u_1,\ldots,u_n) = \psi_{n-1}(u_1,\ldots,u_{n-1})u_n,
\end{equation}
where $\psi_{n-1}(u_1,\ldots,u_{n-1})$ is any function of $(u_1,\ldots,u_{n-1})$ taking values in $\{-1,+1\}$ and $\psi_0\in\{-1,+1\}$ .
\end{proposition}
\end{mdframed}

We note that the map on $\{-1,+1\}^n$ defined by \eqref{genphi} is invertible; that is for any $(v_1,\ldots,v_n)\in\{-1,+1\}^n$, there exists a unique $(u_1,\ldots,u_n)\in\{-1,+1\}^n$ such that $v_1=\psi_0u_1$ and for any $1<k\le n$, $v_k=\psi_{k-1}(u_1,\ldots,u_{k-1})u_k$.
It follows that the filtrations generated by $(\eta_n)_{n\ge1}$ and $(\xi_n)_{n\ge1}$ are identical.

Next, we parametrise the functions $\psi_n$ in terms of the $\max$ function as a ``building block''. To this end, we introduce the following notations. $\mathbb{K}(n)=\mathcal{P}(\{1,\ldots,n\})$ denotes the power set of $\{1,\ldots,n\}$ and for $(u_1,\ldots,u_n)\in\{-1,+1\}^n$:
$$u_{[\emptyset]}=-1\mbox{ and for }K\in\mathbb{K}(n)\setminus\{\emptyset\}, u_{[K]}=\max_{k\in K}u_k.$$
\begin{mdframed}
[style=MyFrame1]
\begin{proposition}\label{genphimaxth}
A function $\psi$ on $\{-1,+1\}^n$ takes values in $\{-1,+1\}$ if and only if it can be written as
\begin{equation}\label{genphimax}
\psi(u_1,\ldots,u_n) = \prod_{K\in\mathbb{K}(n)}u_{[K]}^{\beta_K}
\end{equation}
where $\beta_K\in\{0,1\}$.
Furthermore, this representation is unique.
\end{proposition}
\end{mdframed}
\begin{proof} We first establish the uniqueness of the representation. To this end we must solve the system of equations
\begin{equation}\label{uniqbeta}
\prod_{K\in\mathbb{K}(n)}u_{[K]}^{\beta_K} = \prod_{K\in\mathbb{K}(n)}u_{[K]}^{\beta'_K},\ \forall u_1,\ldots,u_n.
\end{equation}
Solving the equation for $u_1=\dots=u_n=1$ immediately yields $\beta_\emptyset=\beta_\emptyset'$.

Let $\mathbb{B}=\{K\in\mathbb{K}(n):\beta_K\neq\beta_K'\}$. We proceed by contradiction. Assume that $\mathbb{B}\neq\emptyset$ and let $K^*$ be a minimal element in $\mathbb{B}$ (i.e. $K^*\in\mathbb{B}$ and either $K^*$ is a singleton or for any $K\subsetneq K^*$, $\beta_K=\beta_K'$). We choose $u=(u_1,\dots,u_{n-1})$ such that $u_k=-1$ for $k\in K^*$ and $u_k=1$ for $k\not\in K^*$. For such $u$, $u_{[K]}=1$ unless $K\subset K^*$, and \eqref{uniqbeta} reduces to
$$\prod_{\substack{K\in\mathbb{K}(n)\\[1pt]K\subset K^*}}u_{[K]}^{\beta_K} = \prod_{\substack{K\in\mathbb{K}(n)\\[1pt]K\subset K^*}}u_{[K]}^{\beta'_K}.$$
Since $K^*$ is minimal in $\mathbb{B}$, we deduce that we must have $\beta_{n,K^*}=\beta'_{n,K^*}$, which contradicts the assumption that $K^*\in\mathbb{B}$.

To establish that every function on $\{-1,+1\}^n$ taking values in $\{-1,+1\}$ is of the form \eqref{genphimax}, we show that the two sets of functions have the same cardinality.
Indeed, there are $2^{2^n}$ functions $\psi:\{-1,+1\}^n\longrightarrow\{-1,+1\}$ and there are as many choices of $\beta_K$, $K\in\mathbb{K}(n)$.
\end{proof}

While the use of the $\max$ function in \eqref{genphimax} is natural (and in some way canonical), it is not the only ``building block'' one can use to represent functions $\psi:\{-1,+1\}^n\longrightarrow\{-1,+1\}$. We present here a generic way for constructing a representation of the form \eqref{genphimax}. We start by labeling all elements of $\{-1,+1\}^n$. We do so with the use of the sets $K\in\mathbb{K}(n)$ so that elements of $\{-1,+1\}^n$ are written as $u_K$, $K\in\mathbb{K}(n)$. Next we choose a partial order, $\prec$, on $\mathbb{K}(n)$ (which in turn induces a partial order on $\{-1,+1\}^n$). Finally, we let $g_K$, $K\in\mathbb{K}(n)$, be the family of functions on $\{-1,+1\}^n$ taking values in $\{-1,+1\}$, such that for any $K,K'\in\mathbb{K}(n)$, $g_K(u_{K'})=-1$ if and only if $K\prec K'$. Then any function $\psi$ on $\{-1,+1\}^n$ taking values in $\{-1,+1\}$ can be written as
$$\psi(u) = \prod_{K\in\mathbb{K}(n)}g_K(u)^{\beta_K},$$
where $\beta_K\in\{0,1\}$. This representation is unique in the sense that there is a one-to-one correspondence between the functions $\psi$ and the sequences $\beta_K$.
The proof of this statement is an immediate adaptation of the proof of Proposition \ref{genphimaxth}.

In other words, any partial order on $\mathbb{K}(n)$ (and labelling of $\{-1,+1\}^n$) defines a new set of building blocks $g_K$, $K\in\mathbb{K}(n)$, that can be used to produce a representation of the form \eqref{genphimax}.

The $\max$ function is obtained by labeling $u\in\{-1,+1\}^n$ with the set of indices carrying the value $-1$, $K=\{k:u_k=-1\}$, and using the usual inclusion as a partial order on $\mathbb{K}(n)$.

The $\min$ function can also be used as can easily be seen by replacing $u$ with $-u$ in the $\max$ representation, and making all necessary adjustments.

Another example can be constructed by thinking of $\mathbb{K}(n)$ (or equivalently of $\{-1,+1\}^n$) as a totally unordered set. Then one can simply use the functions
$$g_K(u) = \left\{
\begin{array}{ll}
-1 & \mbox{if }u=u_K\\
1 & \mbox{if }u\neq u_K
\end{array}\right.$$

The dependence between the random variables $\xi_{[K]}$, for various $K$'s, renders the computation of quantities such as $\mathbb{E}\Big[\prod_{K\in\mathbb{K}(n)}\xi_{[K]}^{\beta_K}\Big]$ cumbersome. The next proposition enables a linearisation of the product $\prod_{K\in\mathbb{K}(n)}\xi_{[K]}^{\beta_K}$ and therefore a better handle on its expectation.
\begin{mdframed}
[style=MyFrame1]
\begin{proposition}\label{genphilinearprop}
For any collection of sets of integers $M_1,\ldots,M_m$,
$$\prod_{k=1}^mu_{[M_k]} = \frac12(-1)^m - \frac12\sum_{k=0}^m(-2)^k\sum_{\substack{K\in\mathbb{K}(m)\\[1pt]|K|=k}}u_{[M_{[K]}]},$$
where $M_{[K]}=\bigcup_{j\in K}M_j$ and $M_{[\emptyset]}=\emptyset$.
\end{proposition}
\end{mdframed}
\begin{proof}
The identity is clearly true for $m=1$:
$\displaystyle -\frac12 + \frac12 - \frac12(-2)u_{[M_1]} = u_{[M_1]}$.
Suppose the identity true for $m-1$. Then,
\begin{eqnarray*}
\lefteqn{\prod_{k=1}^mu_{[M_k]} = \left(\prod_{k=1}^{m-1}u_{[M_k]}\right)u_{[M_m]}}\\
& = & \frac12(-1)^{m-1}u_{[M_m]} - \frac12\sum_{k=0}^{m-1}(-2)^k\sum_{\substack{K\in\mathbb{K}(m-1)\\[1pt]|K|=k}}u_{[M_{[K]}]}u_{[M_m]}\\
& = & \frac12(-1)^{m-1}u_{[M_m]} - \frac12\sum_{k=0}^{m-1}(-2)^k\sum_{\substack{K\in\mathbb{K}(m-1)\\[1pt]|K|=k}}\left(1 + u_{[M_{[K]}]} + u_{[M_m]} - 2u_{[M_{[K]}\cup M_m]}\right).
\end{eqnarray*}
Using the facts that
$$\sum_{k=0}^{m-1}(-2)^k{m-1\choose k} = (-1)^{m-1}$$
and any $K\in\mathbb{K}(m)$ is either in $\mathbb{K}(m-1)$ or is of the form $K'\cup\{m\}$, where $K'\in\mathbb{K}(m-1)$, we see that
\begin{eqnarray*}
\prod_{k=1}^mu_{[M_k]} & = & -\frac12(-1)^{m-1} - \frac12\sum_{k=0}^{m-1}(-2)^k\sum_{\substack{K\in\mathbb{K}(m-1)\\[1pt]|K|=k}}\left(u_{[M_{[K]}]} - 2u_{[M_{[K]}\cup M_m]}\right)\\
& = & \frac12(-1)^m - \frac12\sum_{k=0}^m(-2)^k\sum_{\substack{K\in\mathbb{K}(m)\\[1pt]|K|=k}}u_{[M_{[K]}]}.
\end{eqnarray*}
\end{proof}

Combining Propositions \ref{genphiprop}, \ref{genphimaxth} and \ref{genphilinearprop}, we obtain the following theorem.
\begin{mdframed}
[style=MyFrame1]
\begin{theorem}\label{genphith}
Let $\eta_n=\phi_n(\xi_1,\dots,\xi_n)$, then $(\eta_n)_{n\ge1}\stackrel{d}{=}(\xi_n)_{n\ge1}$ if and only if for each $n$ and each $K\in\mathbb{K}(n-1)$, there exists $\beta_{n,K}\in\{0,1\}$ such that
\begin{equation}\label{genphibeta}
\phi_n(u_1,\ldots,u_n) = \left(\prod_{K\in\mathbb{K}(n-1)}u_{[K]}^{\beta_{n,K}}\right)u_n.
\end{equation}
Furthermore, for such functions
$$\phi_n(u_1,\ldots,u_n) = \left(\frac12(-1)^{|\mathcal{B}(n)|} - \frac12\sum_{H\in\mathcal{P}(\mathcal{B}(n))}(-2)^{|H|}u_{[\left<H\right>]}\right)u_n.$$
where $\mathcal{B}(n) = \{K\in\mathbb{K}(n-1):\beta_{n,K}=1\}$ and for $H=\{K_1,\ldots,K_h\}\in\mathcal{P}(\mathcal{B}(n))$, $\left<H\right>=\bigcup_{j=1}^hK_j$.
\end{theorem}
\end{mdframed}
For any set of integers $M$, we let $q(M) = \mathbb{P}(\xi_{[M]}=-1) = 2^{-|M|}$. Then $\mathbb{E}[\xi_{[M]}] = 1-2q(M)$.
\begin{mdframed}
[style=MyFrame1]
\begin{corollary}
Let $\zeta_{k-1} = \eta_k\xi_k = \prod_{K\in\mathbb{K}(k-1)}\xi_{[K]}^{\beta_{k,K}}$. Then for any $k\geq2$,
$$\mathbb{E}[\zeta_{k-1}] = \sum_{H\in\mathcal{P}(\mathcal{B}(k))}(-2)^{|H|}q(\left<H\right>)$$
and any $k\neq\ell$,
$$\mathbb{E}[\zeta_{k-1}\zeta_{\ell-1}] = \sum_{\substack{H\in\mathcal{P}(\mathcal{B}(k))\\[1pt]J\in\mathcal{P}(\mathcal{B}(\ell))}}(-2)^{|H|+|J|}q(\left<H\right>\cup\left<J\right>).$$
\end{corollary}
\end{mdframed}
\section{The Gaussian case}\label{Section:Gauss}

Let $\displaystyle U^{(n)}_t=\frac1{\sqrt{n}}X_{\lfloor nt\rfloor}\mbox{ and }V^{(n)}_t=\frac1{\sqrt{n}}Y_{\lfloor nt\rfloor}$. It is well-known that $U^{(n)}_t$, $t\in[0,1]$, converges weakly to a standard Brownian motion, and the same is also trivially true for $V^{(n)}_t$, $t\in[0,1]$. In this section we give necessary and sufficient conditions on the parameters of the GBRW for the two-dimensional process $W^{(n)}_t=(U^{(n)}_t,V^{(n)}_t)$, $t\in[0,1]$, to converge to a two-dimensional Brownian motion.

\subsection{The main result}
\begin{mdframed}
[style=MyFrame1]
\begin{theorem}\label{mainth}
$W^{(n)}$ converges weakly to a two-dimensional Brownian motion (possibly degenerate) with correlation $\rho$ if and only if
\begin{enumerate}
\item[(A)] $\bds\rho=\lim_{n\to\infty}\frac1n\sum_{k=1}^n\sum_{H\in\mathcal{P}(\mathcal{B}(k))}(-2)^{|H|}q(\left<H\right>)\eds$ exists;
\item[(B)] $\bds\lim_{n\to\infty}\frac1{n^2}\sum_{\ell=1}^n\sum_{k=1}^n\sum_{\substack{H\in\mathcal{P}(\mathcal{B}(k))\\[1pt]J\in\mathcal{P}(\mathcal{B}(\ell))}}(-2)^{|H|+|J|}q(\left<H\right>\cup\left<J\right>) = \rho^2\eds$.
\end{enumerate}
\end{theorem}
\end{mdframed}
\begin{proof}
For each $n$, $U^{(n)}$ and $V^{(n)}$ are martingales with respect to $\mathcal{F}^{(n)}_t=\sigma(\xi_1,\ldots,\xi_{\lfloor nt\rfloor})$. Furthermore
$$\Delta U^{(n)}_t = \left\{
\begin{array}{ll}
0 & nt\not\in\mathbb{N}\\
\frac1{\sqrt{n}}\xi_{nt} & nt\in\mathbb{N}
\end{array}\right.,\ |\Delta U^{(n)}_t| \leq \frac1{\sqrt{n}}|\xi_{\lfloor nt\rfloor}| \leq 1$$
and similarly for $V^{(n)}_t$.

Using Theorem VIII.3.11 of \cite{JacodShiryaev}, we see that $W^{(n)}$ converges weakly to a two-dimensional Brownian motion with correlation $\rho$ if and only if $[U^{(n)},V^{(n)}]_t$ converges in probability to $\rho t$.
However,
$$[U^{(n)},V^{(n)}]_t = \sum_{s\leq t}\Delta U^{(n)}_s\Delta V^{(n)}_s = \frac1n\sum_{k=1}^{\lfloor nt\rfloor}\xi_k\eta_k = \frac1n\sum_{k=1}^{\lfloor nt\rfloor}\zeta_{k-1}.$$

Suppose $[U^{(n)},V^{(n)}]_t$ converges in probability to $\rho t$. Since $|[U^{(n)},V^{(n)}]_t|\leq\lfloor nt\rfloor/n\leq1$, we must have
$$\rho=\lim_n\mathbb{E}\big[[U^{(n)},V^{(n)}]_1\big]=\lim_n\frac1n\sum_{k=1}^n\mathbb{E}[\zeta_{k-1}]=\lim_n\frac1n\sum_{k=1}^n\sum_{H\in\mathcal{P}(\mathcal{B}(k))}(-2)^{|H|}q(\left<H\right>),$$
from which we deduce that (A) must hold. We must also have
\begin{eqnarray*}
\rho^2 & = & \lim_n\mathbb{E}\big[[U^{(n)},V^{(n)}]_1^2\big]\ =\ \lim_n\frac1{n^2}\sum_{k=1}^n\sum_{\ell=1}^n\mathbb{E}[\zeta_{k-1}\zeta_{\ell-1}]\\
& = & \lim_n\frac1{n^2}\sum_{k=1}^n\sum_{\ell=1}^n\sum_{\substack{H\in\mathcal{P}(\mathcal{B}(k))\\[1pt]J\in\mathcal{P}(\mathcal{B}(\ell))}}(-2)^{|H|+|J|}q(\left<H\right>\cup\left<J\right>)
\end{eqnarray*}
and (B) must also hold proving the necessity of these conditions.

Next we show sufficiency. In fact we shall prove that (A) and (B) lead to an $L^2$ convergence of $[U^{(n)},V^{(n)}]_t$ to $\rho t$.

Let
$$\rho_k = \mathbb{E}[\zeta_{k-1}] = \sum_{H\in\mathcal{P}(\mathcal{B}(k))}(-2)^{|H|}q(\left<H\right>)$$
and
$$\theta_{k,\ell} = \mathbb{E}[\zeta_{k-1}\zeta_{\ell-1}] = \sum_{\substack{H\in\mathcal{P}(\mathcal{B}(k))\\[1pt]J\in\mathcal{P}(\mathcal{B}(\ell))}}(-2)^{|H|+|J|}q(\left<H\right>\cup\left<J\right>).$$
Then, assuming (A) and (B),
\begin{eqnarray*}
\lefteqn{\mathbb{E}\left[\left(\frac1n\sum_{k=1}^{\lfloor nt\rfloor}\zeta_{k-1}-\rho t\right)^2\right]}\\
& = & \frac1{n^2}\sum_{\ell=1}^{\lfloor nt\rfloor}\sum_{k=1}^{\lfloor nt\rfloor}\mathbb{E}[\zeta_{k-1}\zeta_{\ell-1}] - \frac{2\rho t}n\sum_{k=1}^{\lfloor nt\rfloor}\mathbb{E}[\zeta_{k-1}] + \rho^2t^2\\
& = & \frac1{n^2}\sum_{\ell=1}^{\lfloor nt\rfloor}\sum_{k=1}^{\lfloor nt\rfloor}\theta_{k,\ell} - \frac{2\rho t}n\sum_{k=1}^{\lfloor nt\rfloor}\rho_k + \rho^2t^2\ \underset{n\uparrow\infty}{\longrightarrow}\ \rho^2t^2 - 2\rho^2t^2 + \rho^2t^2\ =\ 0
\end{eqnarray*}
\end{proof}

The asymptotics of $W^{(n)}$ is dependent on the (asymptotic) behaviour of the families $\mathcal{B}(n)$. The following example provides settings in which the latter is easily described.

\begin{example}\label{predictable}
Suppose $\eta_k=\psi_{k-2}(\xi_1,\ldots,\xi_{k-2})\xi_{k-1}\xi_k$, where $\psi_{k-2}$ is any function on $\{-1,+1\}^{k-2}$ taking values in $\{-1,+1\}$. In this case, for any $k<\ell$,
$$\mathbb{E}[\zeta_{k-1}] = 0 \text{ and } \mathbb{E}[\zeta_{k-1}\zeta_{\ell-1}] =
\mathbb{E}[\psi_{k-2}(\xi_1,\ldots,\xi_{k-2})\xi_{k-1}
\psi_{\ell-2}(\xi_1,\ldots,\xi_{\ell-2})]\mathbb{E}[\xi_{\ell-1}] = 0.$$
It immediately follows that $W^{(n)}$ converges weakly to a pair of independent Brownian motions.
\end{example}

\subsection{The Extended Bootstrap Random Walk}
We extend the model introduced in \cite{CHS2016} and \cite{CHL2019} to the case
$\displaystyle\eta_n = \xi_n\prod_{k\in M_n}\xi_k,\ n\geq1$, where $M_n$ is not necessarily the entire set $\{1,\ldots,n-1\}$, but any subset thereof. We investigate the convergence of $W^{(n)}$ in terms of the behaviour of the sets $M_n$. Of particular interest is the case of consecutive indexes, $M_n=\{1,\ldots,\floor{R(n)}\}$, for some real function $R$. Note that the case $M_n=\{\floor{r(n)},\ldots,n-1\}$ is covered by Example \ref{predictable}.
\begin{mdframed}
[style=MyFrame1]
\begin{proposition}\label{prop:disjoint}
Suppose $\mathcal{B}(n)$ is made up of disjoint sets of equal cardinality ($\forall K_1,K_2\in\mathcal{B}(n)$, $|K_1|=|K_2|$ and $K_1\cap K_2=\emptyset$, whenever $K_1\neq K_2$). Call $\kappa$ the cardinality.
\begin{enumerate}
\item If $\kappa=1$ or if $\kappa>1$ and $\lim_n|\mathcal{B}(n)|=+\infty$, then (A) holds true with $\rho=0$.
\item If $\kappa>1$ and $\lim_n|\mathcal{B}(n)|=m<+\infty$, then (A) holds true with $\rho=\left(1-2^{1-\kappa}\right)^m$.
\end{enumerate}
\end{proposition}
\end{mdframed}
\begin{proof}
For $H\in\mathcal{P}(\mathcal{B}(n))$, $|\left<H\right>|=\sum_{K\in H}|K|=\kappa|H|$. It follows that for such an $H$, $q(\left<H\right>)=2^{-|\left<H\right>|}=2^{-\kappa|H|}$ and
\begin{eqnarray*}
\lefteqn{\sum_{H\in\mathcal{P}(\mathcal{B}(n))}(-2)^{|H|}q(\left<H\right>)}\\
& = & \sum_{H\in\mathcal{P}(\mathcal{B}(n))}(-2)^{|H|}2^{-\kappa|H|}\ =\ \sum_{H\in\mathcal{P}(\mathcal{B}(n))}\left(-2^{1-\kappa}\right)^{|H|}\\
& = & \sum_{h=0}^{|\mathcal{B}(n)|}\sum_{\substack{H\in\mathcal{P}(\mathcal{B}(n))\\[1pt]|H|=h}}\left(-2^{1-\kappa}\right)^h\ =\ \sum_{h=0}^{|\mathcal{B}(n)|}{|\mathcal{B}(n)|\choose h}\left(-2^{1-\kappa}\right)^h\
=\ \left(1-2^{1-\kappa}\right)^{|\mathcal{B}(n)|}
\end{eqnarray*}
The result immediately follows.
\end{proof}

Next we focus on the case $\kappa=1$. We call the resulting process the Extended Bootstrap Random Walk.
\begin{mdframed}
[style=MyFrame1]
\begin{corollary}
Suppose $\zeta_{k-1}=\prod_{j\in M_k}\xi_j$, where $M_k\in\mathbb{K}(k-1)$ (i.e. $\kappa=1$).
\begin{enumerate}
\item If $\displaystyle\lim_n\frac1n\sum_{k=1}^n1_{M_k=M_n}=0$, then $W^{(n)}$ converges weakly to a pair of independent Brownian motions.
\item Suppose further that $M_k\subseteq M_{k+1}$ and let $N(n)=\inf\{k:M_k=M_n\}$.
If $\lim_nN(n)/n=1$, then $W^{(n)}$ converges weakly to a pair of independent Brownian motions.
\end{enumerate}
\end{corollary}
\end{mdframed}
\begin{proof}
\begin{enumerate}
\item Proposition \ref{prop:disjoint} guarantees that (A) holds true with $\rho=0$. We check (B):
$$\frac1{n^2}\sum_{\ell=1}^n\sum_{k=1}^n\mathbb{E}[\zeta_{k-1}\zeta_{\ell-1}] = -\frac1n+\frac2{n^2}\sum_{\ell=1}^n\sum_{k=1}^{\ell}1_{M_k=M_\ell} \le \frac2n\sum_{\ell=1}^n\frac1\ell\sum_{k=1}^{\ell}1_{M_k=M_\ell}.$$
Under the assumption that $\displaystyle\lim_n\frac1n\sum_{k=1}^n1_{M_k=M_n}=0$,
$$\lim_n\frac1{n^2}\sum_{\ell=1}^n\sum_{k=1}^n\mathbb{E}[\zeta_{k-1}\zeta_{\ell-1}] = 0 = \rho^2$$
and $W^{(n)}$ converges weakly to a pair of independent Brownian motions.
\item Since $M_k\subseteq M_{k+1}\subseteq\{1,\ldots,k\}$, $\displaystyle\frac1n\sum_{k=1}^n1_{M_k=M_n}=(n-N(n)+1)/n$. Therefore (B) holds true as soon as $\lim_nN(n)/n=1$.
\end{enumerate}
\end{proof}

\begin{example}
\begin{enumerate}
\item Fix $\lambda\in(0,1)$ and suppose $M_k=\{1,\ldots,\floor{\lambda k}\}$, for $k>1/\lambda$. Then, for $n>1/\lambda$, $N(n)>(\floor{\lambda n}-1)/\lambda$, and $\lim_nN(n)/n=1$.
\item Suppose $M_k=\{1,\ldots,m(k)\}$, where $m(k)=\floor{\ln k}$ for $k\ge3$.
Then, for any $\ell\in\{\floor{e^{n-1}}+1,\ldots,\floor{e^n}\}$, $m(\ell)=n-1$. It follows that for any $\ell\in\{\floor{e^{n-1}}+1,\ldots,\floor{e^n}\}$, $N(\ell)=\floor{e^{n-1}}+1$ and, for $n\ge2$,
$$\frac{N(\floor{e^n})}{\floor{e^n}}= \frac{\floor{e^{n-1}}+1}{\floor{e^n}}<\frac{e^{n-1}+1}{e^n-1}\le\frac{e+1}{e^2-1}<1.$$
However, $m(\floor{e^{n}}+1)=n$ and $N(\floor{e^{n}}+1)=\floor{e^n}+1$. It follows that $N(n)/n$ does not converge. Furthermore, for $\overline{n}=\floor{e^{n+1}}$ and $\underline{n}=\floor{e^{n}}+1$,
$$\frac2{\overline{n}^2}\sum_{\ell=2}^{\overline{n}}\sum_{k=1}^{\ell-1}1_{M_k=M_\ell} =  \frac2{\overline{n}^2}\sum_{\ell=2}^{\overline{n}}(\ell-N(\ell))\ge \frac2{\overline{n}^2}\sum_{\ell=\underline{n}}^{\overline{n}}(\ell-\underline{n}) = \frac{(\overline{n}-\underline{n}+1)(\overline{n}-\underline{n})}{\overline{n}^2}.$$
Since the last term approaches $1-2e^{-1}+e^{-2}>0$, (B) fails and $W^{(n)}$ does not converge to a two-dimensional Brownian motion.
\end{enumerate}
\end{example}
\begin{mdframed}
[style=MyFrame1]
\begin{corollary}
Suppose $\zeta_{k-1}=\prod_{j=1}^{\floor{R(k)}}\xi_j$, where $R$ is a strictly increasing, continuous, unbounded and regularly varying function with the property that $1\le R(z)<z$. If the regular variation index $\alpha$ is positive ($\alpha>0$), then $W^{(n)}$ converges weakly to a pair of independent Brownian motions.
\end{corollary}
\end{mdframed}
\begin{proof}
Here $M_k=\{1,\ldots,\floor{R(k)}\}$.
Let $k=\floor{R^{-1}(\floor{R(n)})}+1$. Then $R^{-1}(\floor{R(n)})<k$,
$\floor{R(n)}<g(k)$ and $\floor{R(k)}=\floor{R(n)}$. Therefore, $N(n)\le\floor{R^{-1}(\floor{R(n)})}+1$.

Now, let $k=\floor{R^{-1}(\floor{R(n)})}-1$. Then $k+1\le R^{-1}(\floor{R(n)})$,
$R(k)<R(k+1)\le\floor{R(n)}$ and $\floor{R(k)}<\floor{R(n)}$. Therefore, $N(n)>\floor{R^{-1}(\floor{R(n)})}-1$.

In total,
$$\floor{R^{-1}(\floor{R(n)})}-1<N(n)\le\floor{R^{-1}(\floor{R(n)})}+1.$$
Since $\alpha>0$, $\forall~\varepsilon>0$, $\exists n^*$ such that $\forall n\ge n^*$,
$$\frac{R((1-\varepsilon)n)}{R(n)}+\frac1{R(n)}<1$$
or equivalently that
$$1-\varepsilon<\frac1nR^{-1}(R(n)-1)<\frac1nR^{-1}(\floor{R(n)})\leq1.$$
It immediately follows that $\lim_nN(n)/n=1$ and therefore that $W^{(n)}$ converges weakly to a pair of independent Brownian motions.
\end{proof}

\begin{remark}
Note that if the function $R$ is slowly varying ($\alpha=0$), then convergence of $N(n)/n$ is not guaranteed as can be seen from the case $R(z)=\ln z$.
\end{remark}

\subsection{The bounded case}
Here we look at the case where the number of sets in $\mathbb{K}(k-1)$ used in the recycling is fixed (equal to $m$).
\begin{mdframed}
[style=MyFrame1]
\begin{corollary}
Suppose $\zeta_{k-1}=\varepsilon_k\prod_{i=1}^m\xi_{[M^{(i)}_k]}$, where $M^{(i)}_k\in\mathbb{K}(k-1)$, $\varepsilon_k\in\{-1,1\}$, and, writing $M^{[K]}_k$ for $\bigcup_{i\in K}M^{(i)}_k$,
\begin{enumerate}
\item[(A$_b$)] for any $K\in\mathbb{K}(m)$, $\bds\gamma(K)=\lim_{n\to\infty}\frac1n\sum_{k=1}^n\varepsilon_kq\big(M^{[K]}_k\big)\eds$ exists;
\item[(B$_b$)] for any $K_1,K_2\in\mathbb{K}(m)$, $\bds\lim_{n\to\infty}\frac1n\sum_{k=1}^nq\big(M^{[K_1]}_k\cap M^{[K_2]}_\ell\big) = 1\eds$.
\end{enumerate}
Then $(U^{(n)},V^{(n)})$ converges weakly to a two-dimensional Brownian motion (possibly degenerate) with correlation $\sum_{K\in\mathbb{K}(m)}(-2)^{|K|}\gamma(K)$.
\end{corollary}
\end{mdframed}
\begin{proof}
Note that $\varepsilon_k=\xi_{[\emptyset]}^{\beta_{k,\emptyset}}$ and $\varepsilon_k=-1$ if and only if $\emptyset\in\mathcal{B}(k)$. Therefore,
$\mathcal{B}(k) = \{M^{(1)}_k,\ldots,M^{(m)}_k\}$ or $\{\emptyset,M^{(1)}_k,\ldots,M^{(m)}_k\}$  depending on whether $\varepsilon_k=1$ or $-1$. It follows that selecting $H\in\mathcal{P}(\mathcal{B}(k))$ reduces to choosing $K\in\mathbb{K}(m)$. If $\varepsilon_k=1$, then
$$\sum_{H\in\mathcal{P}(\mathcal{B}(k))}(-2)^{|H|}q(\left<H\right>) = \sum_{K\in\mathbb{K}(m)}(-2)^{|K|}q\big(M^{[K]}_k\big).$$
If $\varepsilon_k=-1$, then
\begin{eqnarray*}
\lefteqn{\sum_{H\in\mathcal{P}(\mathcal{B}(k))}(-2)^{|H|}q(\left<H\right>)}\\
& = & \sum_{K\in\mathbb{K}(m)}(-2)^{|K|}q\big(M^{[K]}_k\big)+\sum_{K\in\mathbb{K}(m)}(-2)^{|K|+1}q\big(M^{[K]}_k\big)\\
& = & -\sum_{K\in\mathbb{K}(m)}(-2)^{|K|}q\big(M^{[K]}_k\big),
\end{eqnarray*}
and (A$_b$) clearly implies (A).

To show that (B$_b$) implies (B), we write
\begin{eqnarray*}
\lefteqn{\frac1{n^2}\sum_{\ell=1}^n\sum_{k=1}^n\sum_{\substack{H\in\mathcal{P}(\mathcal{B}(k))\\[1pt]J\in\mathcal{P}(\mathcal{B}(\ell))}}(-2)^{|H|+|J|}\left(q(\left<H\right>\cup\left<J\right>)-q(\left<H\right>)q(\left<J\right>)\right)}\\
& = & \frac1{n^2}\sum_{\ell=1}^n\sum_{k=1}^n\sum_{\substack{H\in\mathcal{P}(\mathcal{B}(k))\\[1pt]J\in\mathcal{P}(\mathcal{B}(\ell))}}(-2)^{|H|+|J|}q(\left<H\right>\cup\left<J\right>)\left(1-q(\left<H\right>\cap\left<J\right>)\right),
\end{eqnarray*}
where we use the fact that $q(M_1\cup M_2)q(M_1\cap M_2)=q(M_1)q(M_2)$. We distinguish three cases: $\varepsilon_k=\varepsilon_\ell=1$, $\varepsilon_k=\varepsilon_\ell=-1$ and $\varepsilon_k\varepsilon_\ell=-1$.
However, $\left<H\right>$ and $\left<J\right>$ are not affected by the inclusion of the empty set, and the only change is the increase by one of $|H|$ and $|J|$.
It follows that
\begin{eqnarray*}
\lefteqn{\sum_{\substack{H\in\mathcal{P}(\mathcal{B}(k))\\[1pt]J\in\mathcal{P}(\mathcal{B}(\ell))}}(-2)^{|H|+|J|}q(\left<H\right>\cup\left<J\right>)\left(1-q(\left<H\right>\cap\left<J\right>)\right)}\\
& = & \varepsilon_k\varepsilon_\ell\sum_{K_1,K_2\in\mathbb{K}(m)}(-2)^{|K_1|+|K_2|}q\big(M^{[K_1]}_k\cup M^{[K_2]}_\ell\big)\Big(1-q\big(M^{[K_1]}_k\cap M^{[K_2]}_\ell\big)\Big)
\end{eqnarray*}
and that
\begin{eqnarray*}
\lefteqn{\Big|\frac1{n^2}\sum_{\ell=1}^n\sum_{k=1}^n\sum_{\substack{H\in\mathcal{P}(\mathcal{B}(k))\\[1pt]J\in\mathcal{P}(\mathcal{B}(\ell))}}(-2)^{|H|+|J|}\left(q(\left<H\right>\cup\left<J\right>)-q(\left<H\right>)q(\left<J\right>)\right)\Big|}\\
& \le & 2^{2m}\sum_{K_1,K_2\in\mathbb{K}(m)}\frac1{n^2}\sum_{\ell=1}^n\sum_{k=1}^n\Big(1-q\big(M^{[K_1]}_k\cap M^{[K_2]}_\ell\big)\Big)
\end{eqnarray*}
If we let $\bds a_n=\sum_{\ell=2}^n\sum_{k=1}^{\ell-1}\Big(1-q\big(M^{[K_1]}_k\cap M^{[K_2]}_\ell\big)\Big)\eds$, then
$$\frac{a_{n+1}-a_n}{(n+1)^2-n^2} = \frac1{2n+1}\sum_{k=1}^n\Big(1-q\big(M^{[K_1]}_k\cap M^{[K_2]}_\ell\big)\Big)\longrightarrow0,$$
and the result follows immediately by application of Stolz-Ces\`aro Theorem.
\end{proof}

\begin{example}
Suppose $\eta_k=\xi_{[\emptyset]}^{\beta_{k,\emptyset}}\xi_k=\varepsilon_k\xi_k$ and let $E_n=\{k\le n:\varepsilon_k=-1\}=\{k\le n:\beta_{k,\emptyset}=1\}$. If $p=\lim_n|E_n|/n$ exists, then $W^{(n)}$ converges weakly to a two-dimensional Brownian motion with correlation $\rho=1-2p$. In particular, given any $\rho\in[-1,1]$, one can construct a GBRW that, when suitably normalized, converges to a two-dimensional Brownian motion with correlation $\rho$.
\end{example}
\begin{mdframed}
[style=MyFrame1]
\begin{corollary}
Suppose $\zeta_{k-1}=\xi_{[M_k]}$, where $M_k\in\mathbb{K}(k-1)$.
\begin{enumerate}
\item If $\lim_n|M_n|=+\infty$, then $(U^{(n)},V^{(n)})$ converges to a degenerate Brownian motion.
\item If $\lim_n|M_n|$ exists and is finite, and $\limsup_nM_n=\emptyset$ (no index appears infinitely many times) then $(U^{(n)},V^{(n)})$ converges to a (proper) Brownian motion with correlation $\rho=1-2^{1-m}$, where $m=\lim_n|M_n|$.
\end{enumerate}
\end{corollary}
\end{mdframed}
\begin{proof}
\begin{enumerate}
\item We check that (A) of Theorem \ref{mainth} holds true with $\rho=1$. Indeed, $\mathcal{B}(k)=\{M_k\}$ and
$$\lim_n\frac1n\sum_{k=1}^n\sum_{H\in\mathcal{P}(\mathcal{B}(k))}(-2)^{|H|}q(\left<H\right>) = \lim_n\frac1n\sum_{k=1}^n(1-2q(M_k)) = 1.$$
Similarly,
\begin{eqnarray*}
\lefteqn{\frac1{n^2}\sum_{\ell=1}^n\sum_{k=1}^n\sum_{\substack{H\in\mathcal{P}(\mathcal{B}(k))\\[1pt]J\in\mathcal{P}(\mathcal{B}(\ell))}}(-2)^{|H|+|J|}q(\left<H\right>\cup\left<J\right>)}\\
& = & \frac1{n^2}\sum_{\ell=1}^n\sum_{k=1}^n\left(1-2q(M_k)-2q(M_\ell)+4q(M_k\cup M_\ell)\right)\\
& = & 1 - \frac4n\sum_{k=1}^nq(M_k) + \frac4{n^2}\sum_{\ell=1}^n\sum_{k=1}^nq(M_k\cup M_\ell),
\end{eqnarray*}
and (B) follows. Applying Theorem \ref{mainth} completes the proof.
\item In this case
$\lim_n\sum_{H\in\mathcal{P}(\mathcal{B}(n))}(-2)^{|H|}q(\left<H\right>) = \lim_n(1-2^{1-|M_n|})$ exists and is strictly smaller than 1.
Next we check that (B) holds. Indeed,
\begin{eqnarray*}
\lefteqn{\left|\frac1{n^2}\sum_{k,\ell=1}^n\sum_{\substack{K\in\mathcal{P}(\mathcal{B}(k))\\[1pt]L\in\mathcal{P}(\mathcal{B}(\ell))}}(-2)^{|K|+|L|}q(\left<K\right>\cup\left<L\right>)-
\left(\frac1n\sum_{k=1}^n\sum_{H\in\mathcal{P}(\mathcal{B}(k))}(-2)^{|H|}q(\left<H\right>)\right)^2\right|}\\
& = & \frac1{n^2}\sum_{k,\ell=1}^n\big(1-2q(M_k)-2q(M_\ell)+4q(M_k\cup M_\ell)-(1-2q(M_k)(1-2q(M_\ell)\big)\\
& = & \frac4{n^2}\sum_{k,\ell=1}^n\big(q(M_k\cup M_\ell)-q(M_k)q(M_\ell)\big)\\
& = & \frac4{n^2}\sum_{k,\ell=1}^nq(M_k\cup M_\ell)(1-q(M_k\cap M_\ell))\\
& \leq & \frac4{n^2}\sum_{k,\ell=1}^n(1-q(M_k\cap M_\ell))\\
& = & \frac4{n^2}\sum_{k=1}^n(1-q(M_k))+\frac8{n^2}\sum_{\ell=2}^n\sum_{k=1}^{\ell-1}(1-q(M_k\cap M_\ell))
\end{eqnarray*}
The first term clearly converges to 0. To obtain the limit of the second term, we let $a_n=\sum_{\ell=2}^n\sum_{k=1}^{\ell-1}(1-q(M_k\cap M_\ell))$ and observe that
$$\frac{a_{n+1}-a_n}{(n+1)^2-n^2} = \frac1{2n+1}\sum_{k=1}^n(1-q(M_k\cap M_{n+1})).$$
Since $|M_n|$ becomes constant (for $n$ large enough), applying Lemma \ref{onesetlemma}, we see that
$$\lim_n\frac1n\sum_{k=1}^n(1-q(M_k\cap M_{n+1})) \leq  \lim_n\frac1n\sum_{k=1}^n|M_k\cap M_{n+1}| = 0.$$

The result immediately follows.
\end{enumerate}
\end{proof}

\begin{example}
If $\eta_k=\max(\xi_{k-m},\ldots,\xi_{k-1})\xi_k$, then $|M_k|=m$, $\limsup_nM_n=\emptyset$ and $\rho=1-2^{1-m}$.
\end{example}

\section{The non-Gaussian case}\label{Section:NonGauss}

As stated in the introduction, it is not difficult to imagine situations where $(U^{(n)},V^{(n)})$ converges to a non-Gaussian limit or even fails to converge. In this section we look at a particular setting that leads to non-degenerate non-Gaussian limits.

First, we make the observation that $(U^{(n)})_{n\ge0}$ and $(V^{(n)})_{n\ge0}$ are tight and therefore so is $((U^{(n)},V^{(n)}))_{n\ge0}$. Indeed, $(U^{(n)})_{n\ge0}$ and $(V^{(n)})_{n\ge0}$ are weakly convergent sequences, and as such we know that for any $\varepsilon>0$ there exist
compact sets $C_U$ and $C_V$ such that $\sup_n\mathbb{P}(U^{(n)}\not\in C_U)<\varepsilon/2$ and $\sup_n\mathbb{P}(V^{(n)}\not\in C_V)<\varepsilon/2$, from which we get
\begin{eqnarray*}
\mathbb{P}((U^{(n)},V^{(n)})\not\in C_U\times C_V) & = & \mathbb{P}(U^{(n)}\not\in C_U\mbox{ or }V^{(n)}\not\in C_V)\\
& \leq & \mathbb{P}(U^{(n)}\not\in C_U)+\mathbb{P}(V^{(n)}\not\in C_V)\ <\ \varepsilon.
\end{eqnarray*}
Since the product of two compact sets is a compact set in the product space, we deduce that $((U^{(n)},V^{(n)}))_{n\ge0}$ is tight, and by Prohorov's Theorem, it contains a weakly convergent subsequence.

The question we ask is what settings lead to convergence (of the whole sequence). Amongst those is the case of symmetric functions. A function is said to be symmetric if it is unchanged by any permutation of the coordinates. When dealing with functions on $\{-1,1\}^n$, these can be described in a succinct manner.
\begin{mdframed}
[style=MyFrame1]
\begin{lemma}
A function $\psi$ on $\{-1,1\}^n$ is symmetric if and only if there exists a function $f$ such that
$$\psi(u_1,\ldots,u_n) = f(s_n),$$
where $s_n=\sum_{k=1}^nu_k$.
\end{lemma}
\end{mdframed}
\begin{proof}
Let $\nu_n$ be the number of components equal to $-1$ in $(u_1,\ldots,u_n)$: $\nu_n=|\{k:u_k=-1\}|$. Since $\psi$ is symmetric,
$$\psi(u_1,\ldots,u_n) = \psi(\underbrace{-1,\ldots,-1}_{\nu_n},\underbrace{+1,\ldots,+1}_{n-\nu_n}).$$
Therefore $\psi(u)$ depends on $(u_1,\ldots,u_n)$ only through $\nu_n$. The result follows from the observation that $s_n = n-2\nu_n$.
\end{proof}
For example, in the special case of the product (studied in \cite{CHS2016} and \cite{CHL2019}), the $\psi_n$'s are symmetric and $\prod_{k=1}^nu_k = (-1)^{(n-s_n)/2}$.

We wish to investigate the asymptotics of $(U^{(n)},V^{(n)})$ in the case of symmetric $\psi_n$'s. More specifically, we assume that
$$\psi_{k-1}(u_1,\ldots,u_{k-1})=f((u_1+\ldots+u_{k-1})/\sqrt{k}),$$
for some right-continuous function $f$ taking values in $\{-1,1\}$. In other words, we assume that
$$\eta_k = \phi_k(\xi_1,\ldots,\xi_k) = \psi_{k-1}(\xi_1,\ldots,\xi_{k-1})\xi_k = f(X_{k-1}/\sqrt{k})\xi_k.$$
Since
$$\Delta U^{(n)}_t=U^{(n)}_t-U^{(n)}_{t-}=\left\{
\begin{array}{ll}
\xi_k/\sqrt{n} & t=k/n,\ k=1,\ldots,n\\
0 & \text{otherwise}
\end{array}
\right.$$
the above assumptions lead to
$$V^{(n)}_t = \frac1{\sqrt{n}}\sum_{k=1}^{\floor{nt}}f(X_{k-1}/\sqrt{k})\xi_k = \int_{(0,t]}f(U^{(n)}_{s-}/\sqrt{s})dU^{(n)}_s.$$
We shall further assume that the set $\Delta_f=\{z\in\mathbb{R}:\Delta f(z)=f(z)-f(z-)\neq0\}$ is non-empty and has positive minimum gap; that is $\delta_0=\inf_{a\in\Delta_f}\inf_{b\in\Delta_f\setminus\{a\}}|b-a|>0$.
Note that if $a\in\Delta_f$, $\Delta f(a)=2f(a)$ and that if $a,b\in\Delta_f$ and $(a,b)\cap\Delta_f=\emptyset$, then $f(b)=-f(a)$.

For $a\in\Delta_f$ and $0<\delta<\delta_0$, let
$$f^{(a)}_\delta(z) = f(a)\left(\frac1\delta(z-a)1_{|z-a|<\delta}+\sgn(z-a)1_{|z-a|\ge\delta}\right).$$
Next, we introduce a family, indexed by $\delta\in(0,\delta_0)$, of continuous piecewise linear functions that approximate $f$:
$$f_\delta(z) = \left\{
\begin{array}{ll}
f^{(a)}_\delta(z) & z\in(a-\delta,a+\delta),\ a\in\Delta_f\\
f(z) & z\not\in\bigcup_{a\in\Delta_f}(a-\delta,a+\delta)
\end{array}
\right.$$
\begin{mdframed}
[style=MyFrame1]
\begin{lemma}\label{fdelta}
$\forall \delta\in(0,\delta_0)$, $\forall z\in\mathbb{R}$, $(f_\delta(z)-f(z))^2<\sum_{a\in\Delta_f}1_{(a-\delta,a+\delta)}(z)$. In particular,
$\forall z\not\in\Delta_f$, $\lim_{\delta\to0}f_\delta(z)=f(z)$. Furthermore, $|f_\delta-f|$ is a continuous function.
\end{lemma}
\end{mdframed}

\begin{mdframed}
[style=MyFrame1]
\begin{theorem}
If $\eta_k = f(X_{k-1}/\sqrt{k})\xi_k$ where the right-continuous function $f$ has positive minimum gap, then $W^{(n)}$ converges weakly to the non-Gaussian process $\bds \Big(B_t,\int_0^tf(B_s/\sqrt{s})dB_s\Big)_{t\in[0,1]}\eds$.
\end{theorem}
\end{mdframed}
\begin{proof}
As already observed, the first and second components of the process $W$ are one-dimensional Brownian motions.
However, the two-dimensional process is not a two-dimensional Brownian motion because the co-variation process
\begin{equation}
\Big\langle B,\int_0^{\cdot}f(B_s/\sqrt{s})dB_s\Big\rangle_t =\int_0^tf(B_s/\sqrt{s})ds
\end{equation}
is not a constant times $t$, as it would be for a two-dimensional Brownian motion.

By the Skorokhod representation theorem, the convergence of $U^{(n)}$ to a Brownian motion can be assumed to be almost sure, uniformly on $[0,1]$. Call $B$ the limiting process. We shall establish that for each $t\ge0$,
$$\int_{(0,t]}f(U^{(n)}_{s-}/\sqrt{s})dU^{(n)}_s \overset{\text{prob.}}{\underset{n\uparrow\infty}{\longrightarrow}} \int_0^tf(B_s/\sqrt{s})dB_s.$$
This will immediately imply the joint convergence in probability for any collection of times $t_1<\ldots<t_d$, thus establishing that the limiting finite-dimensional distributions are those of the process $\left(B_t,\int_0^tf(B_s/\sqrt{s})dB_s\right)$.

Fix $t\ge0$ and $\varepsilon>0$. By the $L^2$-isometry for stochastic integrals with respect to the Brownian motion,
\begin{eqnarray*}
\lefteqn{\mathbb{E}\left[\left(\int_0^tf_\delta(B_s/\sqrt{s})dB_s-\int_0^tf(B_s/\sqrt{s})dB_s\right)^2\right]}\\
& = & \mathbb{E}\left[\int_0^t\big(f_\delta(B_s/\sqrt{s})-f(B_s/\sqrt{s})\big)^2ds\right]\ =\ \int_0^t\mathbb{E}\left[\big(f_\delta(B_s/\sqrt{s})-f(B_s/\sqrt{s})\big)^2\right]ds
\end{eqnarray*}
where the right member goes to zero as $\delta\to0$, by the bounded convergence theorem and the fact that for each $s\in(0,t]$, $\mathbb{P}(B_s/\sqrt{s}\in\Delta_f) = 0$ so that
$f_\delta(B_s/\sqrt{s}) \longrightarrow f(B_s/\sqrt{s})$ a.s. as $\delta\to0$. Since convergence in $L^2$ implies convergence in probability, there is $\delta_{\varepsilon,1}\in(0,\delta_0)$ such that for all $0<\delta\le\delta_{\varepsilon,1}$,
\begin{equation}\label{deltaone}
\mathbb{P}\left(\left|\int_0^tf_\delta(B_s/\sqrt{s})dB_s-\int_0^tf(B_s/\sqrt{s})dB_s\right|>\varepsilon\right)<\varepsilon.
\end{equation}
Since almost surely, the amount of time that the Brownian motion $B$ spends in $\Delta_f$ has Lebesgue measure zero, there is $\delta_{\epsilon,2}\in(0,\delta_0)$ such that for all $0<\delta \leq \delta_{\epsilon,2}$
\begin{equation}\label{deltatwo}
\sum_{a\in\Delta_f}\mathbb{E}\left[\int_0^t 1_{(a-\delta,a+\delta)}(B_s/\sqrt{s})ds \right] < \varepsilon^3.
\end{equation}
Let $\delta^* =\min(\delta_{\varepsilon, 1}, \delta_{\varepsilon,2})$.
By the continuity of $f_{\delta^*}$, almost surely, $f_{\delta^*}(U^{(n)}_s/\sqrt{s})$ converges uniformly on $[0,t]$ to $f_{\delta^*}(B_s/\sqrt{s})$.
Thus, almost surely, $(U^{(n)}_s, f_{\delta^*}(U^{(n)}_s/\sqrt{s}) )$ converges uniformly on $[0,t ]$ to $(B_s, f_{\delta^*}(B_s/\sqrt{s}))$.
It then follows from \cite{KP91} that this implies that
$\int_{(0,t]}f_{\delta^*}(U^{(n)}_{s-}/\sqrt{s})dU^{(n)}_s$ converges in probability to $\int_0^t f_{\delta^*}(B_s/\sqrt{s})dB_s$ as $n\to \infty$.
Let $n_\epsilon>0$ such that for all $n\geq n_\epsilon$,
\begin{equation}\label{deltathree}
\mathbb{P}\left( \left|\int_{(0,t]}f_{\delta^*}(U^{(n)}_{s-}/\sqrt{s})dU^{(n)}_s-\int_0^t f_{\delta^*}(B_s/\sqrt{s})dB_s\right|>\varepsilon\right)< \varepsilon.
\end{equation}
Combining the above we have that for all $n\geq n_\epsilon$,
\begin{eqnarray*}
& & \mathbb{P}\left( \left|\int_{(0,t]}f(U^{(n)}_{s-}/\sqrt{s})dU^{(n)}_s -\int_0^tf(B_s/\sqrt{s})dB_s \right|>3\varepsilon\right)\\
&&\leq \mathbb{P}\left(\left|\int_{(0,t]}\left(f(U^{(n)}_{s-}/\sqrt{s})-f_{\delta^*}(U^{(n)}_{s-}/\sqrt{s})\right)dU^{(n)}_s\right|>\varepsilon\right)\\
&&\qquad + \mathbb{P}\left(\left|\int_{(0,t]} f_{\delta^*}(U^{(n)}_{s-}/\sqrt{s})dU^{(n)}_s-\int_0^tf_{\delta^*}(B_s/\sqrt{s})dB_s\right|> \varepsilon \right)\\
&& \qquad + \mathbb{P}\left(\left|\int_0^t f_{\delta^*}(B_s/\sqrt{s})dB_s-\int_0^t f(B_s/\sqrt{s})dB_s \right| >\varepsilon\right)\\
&&\leq \frac{1}{\epsilon^2}\mathbb{E}\left[\left|\int_{(0,t]}\left(f(U^{(n)}_{s-}/\sqrt{s})-f_{\delta^*}(U^{(n)}_{s-}/\sqrt{s})\right)dU^{(n)}_s\right|^2\right] + 2\varepsilon\\
&& =  \frac{1}{\epsilon^2}\mathbb{E}\left[\int_{(0,t]}\left|f(U^{(n)}_{s-}/\sqrt{s})-f_{\delta^*}(U^{(n)}_{s-}/\sqrt{s})\right|^2d[U^{(n)},U^{(n)}]_s \right] + 2\varepsilon\\
\end{eqnarray*}
where we used Markov's inequality, \eqref{deltaone} and \eqref{deltathree} for the second inequality, and the It\^o isometry for the equality.
Recall that
$$[U^{(n)},U^{(n)}]_t = \sum_{0<s\le t}(\Delta U^{(n)}_s)^2 = \sum_{k=1}^{\floor{nt}}(\xi_k/\sqrt{n})^2 = \frac{\floor{nt}}{n}$$
and denote by $g$ the function $(f-f_{\delta^*})^2$. Then
\begin{eqnarray*}
\int_{(0,t]}g(U^{(n)}_{s-}/\sqrt{s})d[U^{(n)},U^{(n)}]_s & = & \sum_{0<s\le t}g(U^{(n)}_{s-}/\sqrt{s})(\Delta U^{(n)}_s)^2\ =\ \frac1n\sum_{k=1}^{\floor{nt}}g(X_{k-1}/\sqrt{k})\\
& = & \int_0^{\floor{nt}/n}g\left(\left(U^{(n)}_{s-}/\sqrt{s}\right)\sqrt{ns/(\floor{ns}+1)}\right)ds\\
& \le & \int_0^tg\left(\left(U^{(n)}_{s-}/\sqrt{s}\right)\sqrt{ns/(\floor{ns}+1)}\right)ds
\end{eqnarray*}
and
\begin{eqnarray}
\lefteqn{\mathbb{P}\left(\left|\int_{(0,t]}f(U^{(n)}_{s-}/\sqrt{s})dU^{(n)}_s -\int_0^tf(B_s/\sqrt{s})dB_s \right|>3\varepsilon\right)}\nonumber\\
& \le & \frac{1}{\epsilon^2}\int_0^tg\left(\left(U^{(n)}_{s-}/\sqrt{s}\right)\sqrt{ns/(\floor{ns}+1)}\right)ds+\varepsilon\label{longineq}
\end{eqnarray}
By the continuity of $g$ and the  bounded convergence theorem, as $n\to \infty$,
\begin{equation}
\mathbb{E}\left[\int_0^tg\left(\left(U^{(n)}_{s-}/\sqrt{s}\right)\sqrt{ns/(\floor{ns}+1)}\right)ds\right] \longrightarrow \mathbb{E}\left[\int_0^tg(B_s/\sqrt{s})ds\right].
\end{equation}
The right-hand side can further be evaluated:
$$\mathbb{E}\left[\int_0^tg(B_s/\sqrt{s})ds\right]\le\sum_{a\in\Delta_f}\mathbb{E}\left[\int_0^t1_{(a-\delta^*,a+\delta^*)}(B_s/\sqrt{s})ds\right]<\varepsilon^3,$$
where we used Lemma \ref{fdelta} for the first inequality and \eqref{deltatwo} for the second.

It follows that there is an $n_\varepsilon' \geq n_\varepsilon$ such that for all $n\geq n_\varepsilon'$,
\begin{equation}
\mathbb{E}\left[\int_0^tg\left(\left(U^{(n)}_{s-}/\sqrt{s}\right)\sqrt{ns/(\floor{ns}+1)}\right)ds\right] \leq 2\varepsilon^3.
\end{equation}
Substituting this in \eqref{longineq}, we find that for all $n\geq n_\varepsilon'$,
\begin{equation}
\mathbb{P}\left(\left|\int_{(0,t]}f(U^{(n)}_{s-}/\sqrt{s})dU^{(n)}_s-\int_0^tf(B_s/\sqrt{s})dB_s\right| > 3\varepsilon\right)\leq 2\varepsilon + 2\varepsilon = 4\varepsilon.
\end{equation}
Since $\varepsilon>0$ was arbitrary, the desired convergence in probability follows.
\end{proof}
\begin{mdframed}
[style=MyFrame1]
\begin{corollary}
Suppose $\bds V^{(n)}_t = \frac1{\sqrt{n}}\sum_{k=1}^{\floor{nt}}\sgn(X_{k-1})\xi_k\eds$, where $\sgn(0)=-1$. Then $W^{(n)}$ converges weakly to the two-dimensional non-Gaussian process
$\bds \Big(B_t,\int_0^t\sgn(B_s)dB_s\Big)_{t\in[0,1]}\eds$.
\end{corollary}
\end{mdframed}
\section{Ergodicity of the GBRW transformation -- A characterisation}\label{Section:Ergodic}

In this section we study the ergodicity of the measure-preserving transformation $(X_n)_{n\ge0}\to(Y_n)_{n\ge0}$. More specifically, let $\Omega$ be the space of integer-valued sequences $x=(x_n)_{n\geq0}$ such that $x_0=0$ and for any $n\geq1$, $x_n-x_{n-1}\in\{-1,+1\}$, endowed with the product sigma-field and the probability measure under which the coordinate map is a simple symmetric random walk started at 0. Given a sequence of functions $\psi_n:\{-1,+1\}^n\to\{-1,+1\}$ ($\psi_0$ is a constant taking value -1 or +1), we let $T$ be the transformation on $\Omega$ such that, with $y=T(x)$, $y_n=\sum_{k=1}^nv_k$, where
$v_k=\psi_{k-1}(u_1,\ldots,u_{k-1})u_k$ and $u_k=x_k-x_{k-1}$. We have seen that $T$ is measure-preserving. Here we ask whether $T$ is ergodic.

We have seen that the discrete-time version of the L\'evy transformation
$$\psi_{k-1}(x_1,\ldots,x_{k-1})=\sgn(x_1+\ldots+x_{k-1})$$
is not ergodic (see \cite{Fujita}). Dubins and Smorodinsky \cite{DubinsSmorodinsky} modified this transformation by first letting $\sgn(0)=0$ and then skipping any flat portions of the path thus produced, and showed that such a transformation is ergodic. Next we give a necessary and sufficient condition, in terms of the sequence $(\psi_n)_{n\ge0}$, for the GBRW to be ergodic.

We start by observing that the ergodicity of $T$ is equivalent to that of $\tau=\Delta\circ T\circ\Delta^{-1}$ defined on the space $\Theta=\{-1,+1\}^\mathbb{N}$, where $\Delta:\Omega\to\Theta$ is the difference operator; that is, for $x\in\Omega$, $u=\Delta(x)$ is the sequence in $\{-1,+1\}^\mathbb{N}$ defined by $u_k=x_k-x_{k-1}$ ($u_1=x_1$). The measure that $\tau$ preserves is the Bernoulli measure $\mu$ defined as
$$\mu\big(\{u\in\Theta:\ u_1=\varepsilon_1,\ldots,u_n=\varepsilon_n\}\big)=\frac1{2^n},$$
for any $\varepsilon_1,\ldots,\varepsilon_n\in\{-1,+1\}$.
For each $n\ge1$, we let $\mu_n$ be the measure induced from $\mu$ by the projection $\pi_n$ on $\Theta_n$: $\mu_n(A_n)=\mu(\pi_n^{-1}(A_n))$. We also let $\tau_n$ be the mapping $\tau_n:(u_1,\ldots,u_n)\to(v_1,\ldots,v_n)$ such that $v_k=\psi_{k-1}(u_1,\ldots,u_{k-1})u_k$. We observe that $\tau_n\circ\pi_n=\pi_n\circ\tau$, that $\tau$ and $\tau_n$ are bijections, and deduce that
$$\tau_n^{-1}(\pi_n(A))=\tau_n^{-1}(\pi_n(\tau(\tau^{-1}(A))))=\tau_n^{-1}(\tau_n(\pi_n(\tau^{-1}(A))))=\pi_n(\tau^{-1}(A)).$$
\begin{mdframed}
[style=MyFrame1]
\begin{lemma}
The transformation $\tau$ is ergodic if and only if, for any $n\geq1$, whenever $\tau_n^{-1}(A_n)=A_n$, we must have $\mu_n(A_n)=0$ or $\mu_n(A_n)=1$.
\end{lemma}
\end{mdframed}
\begin{proof}
Let $A$ be be such that $\tau^{-1}(A)=A$. Then we must have that, for any $n\ge1$, $\tau_n^{-1}(\pi_n(A))=\pi_n(\tau^{-1}(A))=\pi_n(A)$. Since $\pi_{n+1}^{-1}(\pi_{n+1}(A)) = \pi_{n+1}(A)\times\Theta \subset \pi_n(A)\times\Theta = \pi_n^{-1}(\pi_n(A))$, the sequence $(\mu_n(\pi_n(A)))_{n\ge1}$ is non-increasing and must be constant, equal to 0 or 1, for $n$ large enough. As $A=\lim_n\pi_n(A)\times\Theta$, we must have that $\mu(A)=0$ or $\mu(A)=1$, and that $\tau$ must be ergodic.

Conversely, if $\tau$ is ergodic and for $n\geq1$, $A_n$ is such that $\tau_n^{-1}(A_n)=A_n$, then $$\tau^{-1}(\pi_n^{-1}(A_n))=\pi_n^{-1}(\tau_n^{-1}(A_n))=\tau_n^{-1}(A_n)\times\Theta=A_n\times\Theta=\pi_n^{-1}(A_n).$$
It follows that $A_n\times\Theta=\pi_n^{-1}(A_n)$ has measure 0 or 1. It immediately follows that $A_n$ itself has measure 0 or 1.
\end{proof}

We are now ready to state the main result of this section.
\begin{mdframed}
[style=MyFrame1]
\begin{theorem}\label{ergodicth}
Let $T$ be the measure-preserving transformation associated with a sequence of functions $(\psi_n)_{n\geq0}$. $T$ is ergodic (i.e. $\tau$ is ergodic) if and only \begin{enumerate}
\item $\psi_0=-1$ and $\forall n\geq1$, $\prod_{(u_1,\ldots,u_n)\in\Theta_n}\psi_n(u_1,\ldots,u_n)=-1$, or equivalently
\item $\psi_0=-1$ and $\forall n\geq1$, $\max(u_1,\ldots,u_n)$ appears in representation \eqref{genphibeta}; that is $\beta_{n+1,\{1,\ldots,n\}}=1$.
\end{enumerate}
\end{theorem}
\end{mdframed}
\begin{proof}
The key idea in the proof is the fact $\tau$ is ergodic if and only if, for any $n\geq1$, orbits of $\tau_n$ have period $2^n$ as otherwise a size $m<2^n$ orbit $O$ defines a proper subset of $\Theta_n$ such that $\tau_n^{-1}(O)=O$. For $n=1$, we clearly require that $\psi_0=-1$. Next we build $\tau_2$ by constructing the orbit started at $(-1,-1)$: $(-1,-1)\to(+1,-\psi_1(-1))\to(-1,-\psi_1(-1)\psi_1(+1))\to(+1,-\psi_1(+1))$. This orbit is of period $2^2$ if and only if $\psi_1(-1)\psi_1(+1)=-1$. We reason by induction and assume that each of $\tau_1,\ldots,\tau_n$ has a single orbit. Let $\theta_1$ be the vector in $\Theta_n$ made up of $n$ $(-1)$'s and $(\theta_1,\ldots,\theta_{2^n})$ be the orbit of $\tau_n$ started at $\theta_1$. Note that $\tau_n(\theta_{2^n})=\theta_1$. Let $\vartheta_1$ be the vector in $\Theta_{n+1}$ made up of $n+1$ $(-1)$'s. Then
$\vartheta_2=(\theta_2,-\psi_n(\theta_1))$, $\vartheta_3=(\theta_3,-\psi_n(\theta_1)\psi_n(\theta_2))$,\ldots,$\vartheta_{2^n}=\left(\theta_{2^n},-\prod_{k=1}^{2^n-1}\psi_n(\theta_k)\right)$ and
$\vartheta_{2^n+1}=\left(\theta_1,-\prod_{k=1}^{2^n}\psi_n(\theta_k)\right)$. The requirement that $\vartheta_{2^n+1}\neq\vartheta_1$ translates to
$$\prod_{(u_1,\ldots,u_n)\in\Theta_n}\psi_n(u_1,\ldots,u_n)=\prod_{k=1}^{2^n}\psi_n(\theta_k)=-1,$$
which establishes the first statement. The second statement is a direct application of representation \eqref{genphimaxth}:
\begin{eqnarray*}
-1 & = & \prod_{\Theta_n}\psi_n(u_1,\ldots,u_n)\ =\ \prod_{\Theta_n}\prod_{K\in\mathbb{K}(n)}u_{[K]}^{\beta_K}\ =\ \prod_{K\in\mathbb{K}(n)}\prod_{\Theta_n}u_{[K]}^{\beta_{n+1,K}}\ =\ \prod_{K\in\mathbb{K}(n)}\left(\prod_{\Theta_n}u_{[K]}\right)^{\beta_{n+1,K}}\\
& = & \prod_{K\in\mathbb{K}(n)}\left((-1)^{2^{n-|K|}}\right)^{\beta_{n+1,K}}\ =\ (-1)^{\beta_{n+1,\{1,\ldots,n\}}}.
\end{eqnarray*}
\end{proof}

We observe that to each transformation $\tau$ on $\Theta$ ($T$ on $\Omega$) corresponds an ergodic transformation $\tilde\tau$ on $\Theta$ ($\tilde{T}$ on $\Omega$) such that these transformations time-shifted (and space-shifted) coincide with high probability; that is
$$\mu\big(\{u:\ \omega_N(\tau(u))=\omega_N(\tilde\tau(u))\}\big)\ge1-2^{-N},$$
where $\omega_N(u)=(u_{n+N})_{n\ge1}$. The transformation $\tilde\tau$ is simply built from $\tau$ by changing the value of $\beta_{n+1,\{1,\ldots,n\}}$ to 1 whenever required.

\subsubsection*{An ergodic (slightly) modified discrete L\'evy transformation.} As already pointed out the discrete L\'evy transformation is not ergodic but a modification that skips flat portions of the path created by an adjustment to the definition of the $\sgn$ function ($\sgn(0)=0$) is ergodic. Here we suggest another (simpler) modification of the discrete L\'evy transformation.

We start by investigating representation \eqref{genphimaxth} for $\sgn(u_1+\ldots+u_n)$. As a latter is a symmetric function, the $\beta_{n,K}$'s must be equal for all subsets of equal size; that is, with a slight modification of the notation, writing $\beta_{n,|K|}$ for $\beta_{n,K}$,
$$\sgn(u_1+\ldots+u_n)=\prod_{K\in\mathbb{K}(n)}u_{[K]}^{\beta_{n,K}}
= \prod_{k=0}^n\bigg(\prod_{|K|=k}u_{[K]}\bigg)^{\beta_{n,k}}.$$
Now suppose $n\ge2m+1$ and let $(u_1,\ldots,u_n)$ be such that $\nu_n=|\{k:u_k=-1\}|=\ell$ for $\ell\le m$, and $s_n=u_1+\ldots+u_n=n-2\nu_n\ge1$. Letting $\ell=0$ yields
$$1 = \prod_{k=0}^n\bigg(\prod_{|K|=k}u_{[K]}\bigg)^{\beta_{n,k}} = (-1)^{\beta_{n,0}};\mbox{ that is }\beta_{n,0}=0.$$
For $\ell=1$, $\prod_{|K|=1}u_{[K]}=-1$ and, for $k\ge2$, $\prod_{|K|=k}u_{[K]}=1$. Therefore for such $(u_1,\ldots,u_n)$,
$$1 = \Bigg(\bigg(\prod_{|K|=1}u_{[K]}\bigg)^{\beta_{n,1}}\Bigg)\Bigg(\prod_{k=2}^n\bigg(\prod_{|K|=k}u_{[K]}\bigg)^{\beta_{n,k}}\Bigg) = (-1)^{\beta_{n,1}};\mbox{ that is }\beta_{n,1}=0.$$
This can be repeated and yields the fact that $\beta_{n,m}=0$ whenever $n\ge2m+1$.

For $m\le n\le2m$, again choosing $(u_1,\ldots,u_n)$ such that $\nu_n=|\{k:u_k=-1\}|=m$, and $s_n=n-2m\le0$, yields
$$-1 = \Bigg(\prod_{k=0}^{\ell_n}\bigg(\prod_{|K|=k}u_{[K]}\bigg)^{\beta_{n,k}}\Bigg)\Bigg(\prod_{k=\ell_n+1}^m\bigg(\prod_{|K|=k}u_{[K]}\bigg)^{\beta_{n,k}}\Bigg)\Bigg(\prod_{k=m+1}^n\bigg(\prod_{|K|=k}u_{[K]}\bigg)^{\beta_{n,k}}\Bigg),$$
where $\ell_n=\floor{(n-1)/2}$.
Since $\beta_{n,k}=0$ for $k\le\ell_n$ and $\prod_{|K|=k}u_{[K]}=1$ for $k\ge m+1$, the above identity becomes
$$-1 = \prod_{k=\ell_n+1}^m\bigg(\prod_{|K|=k}u_{[K]}\bigg)^{\beta_{n,k}} = \prod_{k=\ell_n+1}^m\bigg((-1)^{{m\choose k}}\bigg)^{\beta_{n,k}} = (-1)^{\sum_{k=\ell_n+1}^m{m\choose k}\beta_{n,k}};$$
that is,
$$\sum_{k=\ell_n+1}^m{m\choose k}\beta_{n,k} = 1 \mod2\mbox{ or equivalently }
\beta_{n,m} = 1 + \sum_{k=\ell_n+1}^{m-1}{m\choose k}\beta_{n,k} \mod2.$$
The figure below describes the array $\beta_{n,k}$ where an orange cell indicates that $\beta_{n,k}=0$ and a blue cell that $\beta_{n,k}=1$.
\begin{figure}[h]
	\includegraphics[height=6cm]{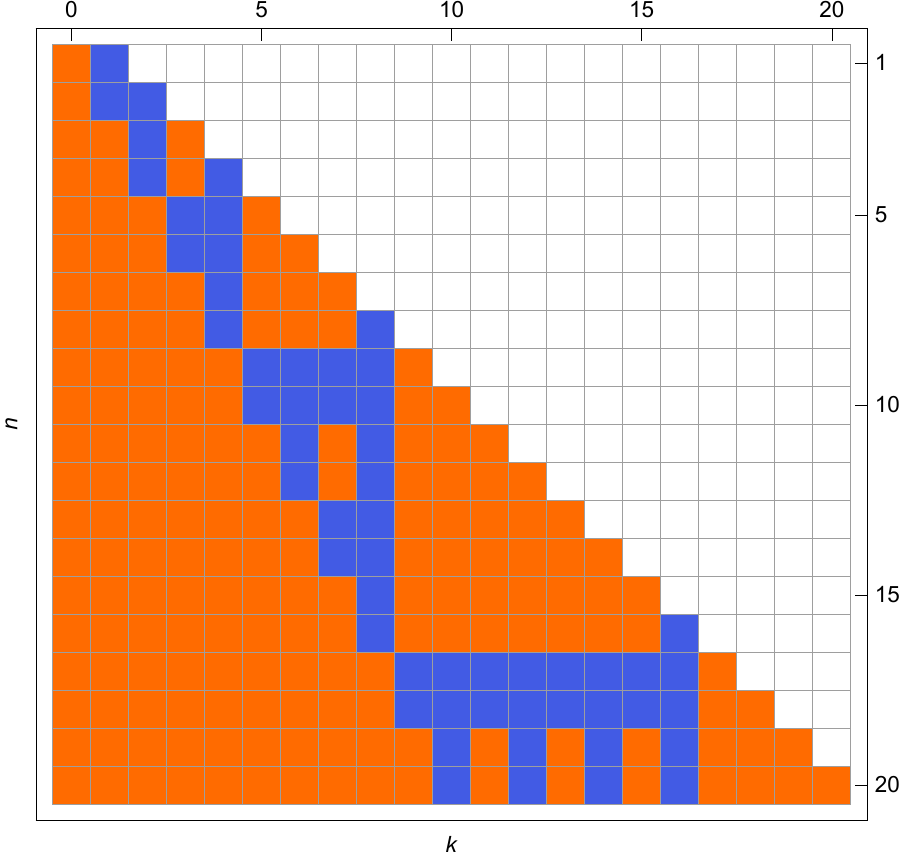}
	\qquad\includegraphics[height=6cm]{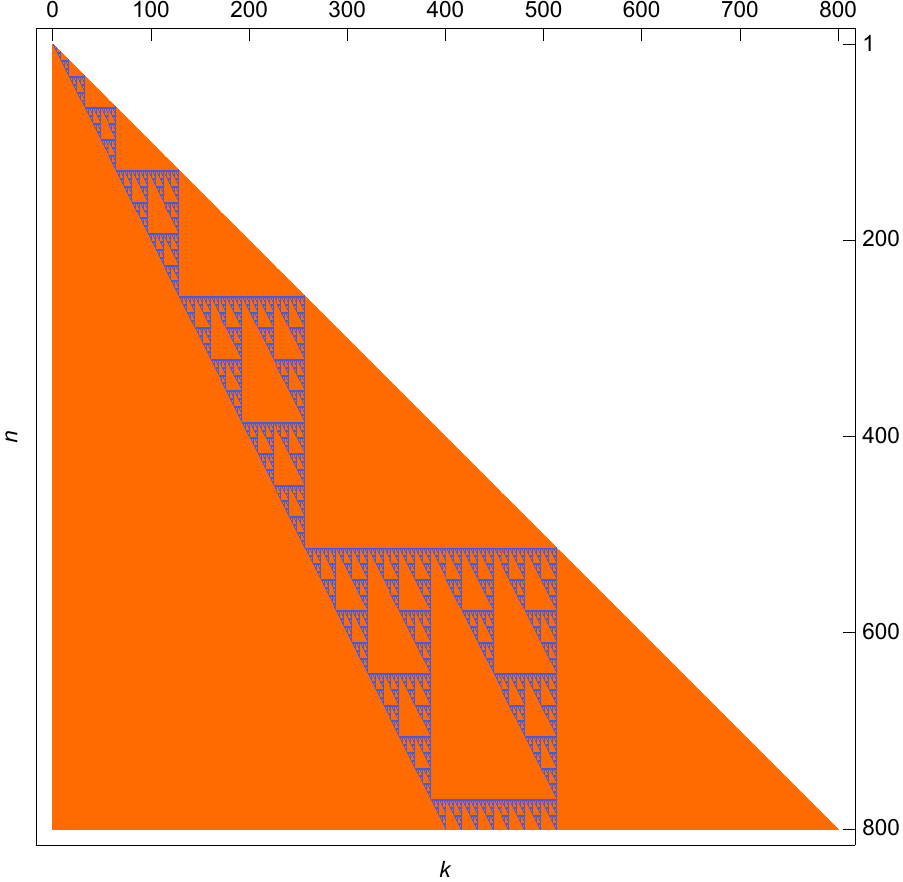}
	\caption{The $\beta_{n,k}$ array for $n\le20$ (left) and $n\le800$ (right)}
\end{figure}

For example,
$$\sgn(u_1+u_2) = u_1u_2\max(u_1,u_2)\mbox{ and }\sgn(u_1+u_2+u_3) = \max(u_1,u_2)\max(u_1,u_3)\max(u_2,u_3).$$

\begin{corollary}
The following slight adaptation of the discrete L\'evy transformation is ergodic: $\psi_0=-1$ and for $n\ge1$, $\psi_n(u_1,\ldots,u_n)=\sgnn_n(u_1+\ldots+u_n)$, where
$$\sgnn_n(s)=
\left\{\begin{array}{ll}
\sgn(s) & n\mbox{ is a power of 2 or }s>-n\\
1 &  n\mbox{ is not a power of 2 and }s\le-n
\end{array}
\right.$$
or equivalently that $\psi_0=-1$ and for $n\ge1$,
$$\psi_n(u_1,\ldots,u_n)=
\left\{\begin{array}{ll}
\sgn(u_1+\ldots+u_n) & n\mbox{ is a power of 2}\\
\max(u_1,\ldots,u_n)\sgn(u_1+\ldots+u_n) &  n\mbox{ is not a power of 2}
\end{array}
\right.$$
\end{corollary}
\begin{proof}
The case $n=1$ is trivial. Fix $n\geq2$,
$$\prod_{(u_1,\ldots,u_n)\in\Theta_n}\sgn(u_1+\ldots+u_n) = \prod_{k=0}^n\big(\sgn(2k-n)\big)^{{n\choose k}} =
\left\{\begin{array}{ll}
(-1)^{\sum_{k=0}^{(n-1)/2}{n\choose k}} & n\mbox{ odd}\\
(-1)^{\sum_{k=0}^{n/2}{n\choose k}} & n\mbox{ even}
\end{array}
\right.$$
However, for $n$ odd, $\sum_{k=0}^{(n-1)/2}{n\choose k} = \sum_{k=(n+1)/2}^n{n\choose k} = 2^{n-1}$ is even. And for $n$ even, $\sum_{k=0}^{n/2}{n\choose k}$ is even if and only if $\frac12{n\choose n/2}$ is even. We use Lucas' Theorem \cite{Lucas} to show that $\frac12{2n\choose n}={2n-1\choose n-1}$ is odd if and only if $n$ is a power of 2.

Suppose that $n=2^\ell$ for $\ell\in\mathbb{N}$. Then
$$n-1=\sum_{k=0}^{\ell-1}2^k\mbox{ and }2n-1=\sum_{k=0}^\ell2^k$$
and every digit in the base 2 expansion of $n-1$ is less than or equal to the corresponding digit in the base 2 expansion of $2n-1$, thus proving that ${2n-1\choose n-1}$ is not divisible by 2.

Now, suppose that $2^\ell<n<2^{\ell+1}$ for $\ell\in\mathbb{N}$ so that
$n=2^\ell+\sum_{k=\kappa+1}^{\ell-1}\alpha_k2^k+2^\kappa$, for $\alpha_{\kappa+1},\ldots,\alpha_{\ell-1}\in\{0,1\}$ and $0\le\kappa<\ell$.

If $\kappa>0$, then
$$n-1=2^\ell+\sum_{k=\kappa+1}^{\ell-1}\alpha_k2^k+\sum_{k=0}^{\kappa-1}2^k\mbox{ and }2n-1=2^{\ell+1}+\sum_{k=\kappa+1}^{\ell-1}\alpha_k2^{k+1}+\sum_{k=0}^\kappa2^k$$
so that the digits in the base 2 expansions of $n-1$ and $2n-1$ are:

\begin{center}
\begin{tabular}{|c|c|c|c|c|c|c|c|c|c|c|}\hline
$k$ & $\ell+1$ & $\ell$ & $\ell-1$ & \ldots & $\kappa+2$ & $\kappa+1$ & $\kappa$ & $\kappa-1$ & \ldots & 0\\\hline
$2n-1$ & 1 & $\alpha_{\ell-1}$ & $\alpha_{\ell-2}$ & \ldots & $\alpha_{\kappa+1}$ & 0 & 1 & 1 & \ldots & 1\\\hline
$n-1$ & 0 & 1 & $\alpha_{\ell-1}$ & \ldots & $\alpha_{\kappa+2}$ & $\alpha_{\kappa+1}$ & 0 & 1 & \ldots & 1\\\hline
\end{tabular}
\end{center}
If $\alpha_{\kappa+1}=\ldots=\alpha_{\ell-1}=1$, then the digit of order $\kappa+1$ in the base 2 expansion of $n-1$ is greater than the corresponding digit in the base 2 expansion of $2n-1$, and ${2n-1\choose n-1}$ is divisible by 2.

Suppose at least one of $\alpha_{\kappa+1},\ldots,\alpha_{\ell-1}$ equals 0 and let  $k^*=\max\{k\in\{\kappa+1,\ldots,\ell-1\}:\ \alpha_k=0\}$. Then the digit of order $k^*+1$ in the base 2 expansion of $n-1$ is greater than the corresponding digit in the base 2 expansion of $2n-1$, and ${2n-1\choose n-1}$ is divisible by 2.

If $\kappa=0$, then
$$n-1=2^\ell+\sum_{k=1}^{\ell-1}\alpha_k2^k\mbox{ and }2n-1=2^{\ell+1}+\sum_{k=1}^{\ell-1}\alpha_k2^{k+1}+1$$
so that the digits in the base 2 expansions of $n-1$ and $2n-1$ are:

\begin{center}
\begin{tabular}{|c|c|c|c|c|c|c|c|}\hline
$k$ & $\ell+1$ & $\ell$ & $\ell-1$ & \ldots & 2 & 1 & 0\\\hline
$2n-1$ & 1 & $\alpha_{\ell-1}$ & $\alpha_{\ell-2}$ & \ldots & $\alpha_1$ & 0 & 1\\\hline
$n-1$ & 0 & 1 & $\alpha_{\ell-1}$ & \ldots & $\alpha_2$ & $\alpha_1$ & 0\\\hline
\end{tabular}
\end{center}
The same argument as before shows that ${2n-1\choose n-1}$ is divisible by 2.
\end{proof}

Another example for which $\beta_{n+1,\{1,\ldots,n\}}=1$ and ergodicity holds, is the following modification of the discrete L\'evy transformation: $$\psi_n(u_1,\ldots,u_n)=\max(u_1,\ldots,u_n)\sgn(u_1+\ldots+u_{n-1}).$$

\newpage

\section*{Appendix}
\vspace*{1em}
\begin{mdframed}
[style=MyFrame1]
\begin{lemma}\label{onesetlemma}
Let $(M_k)_{k\ge1}$ be a sequence of subsets of $\mathbb{N}$. Suppose that the cardinality of $M_k$ is finite and constant. If $\limsup_nM_n = \emptyset$ then
$$\lim_n\frac1n\sum_{k=1}^n|M_k\cap M_{n+1}| = 0.$$
\end{lemma}
\end{mdframed}
\begin{proof}
We prove the contrapositive statement that if $\lim_n\frac1n\sum_{k=1}^n|M_k\cap M_{n+1}| > 0$, then $\limsup_nM_n \neq \emptyset$.

We label, say in the increasing order, the elements of $M_k$, $u^{(1)}_k,\ldots,u^{(m)}_k$, where $m=|M_k|$. Then
$$M_k\cap M_{n+1} = \left(\bigcup_{i=1}^m\{u^{(i)}_k\}\right)\cap\left(\bigcup_{j=1}^m\{u^{(i)}_{n+1}\}\right) = \bigcup_{i,j=1}^m\left(\{u^{(i)}_k\}\cap\{u^{(j)}_{n+1}\}\right)$$
and
$$\sum_{i,j=1}^m\lim_n\frac1n\sum_{k=1}^n|\{u^{(i)}_k\}\cap\{u^{(j)}_{n+1}\}| = \lim_n\frac1n\sum_{k=1}^n|M_k\cap M_{n+1}| > 0,$$
from which we deduce that for at least one pair $(i,j)$, $\lim_n\frac1n\sum_{k=1}^n\delta(u^{(i)}_k,u^{(j)}_{n+1})>0$.
Here $\delta(k,\ell)$ denotes the Kronecker delta function: $\delta(k,\ell)=1$ if $k=\ell$ and 0 otherwise.

Let $A_n=\{k\leq n;u^{(i)}_k=u^{(j)}_{n+1}\}$. Then for $\varepsilon>0$ and $n$ large enough
$$\frac1n|A_n| = \frac1n\sum_{k=1}^n\delta(u^{(i)}_k,u^{(j)}_{n+1}) > \varepsilon;\mbox{ i.e. }|A_n|>n\varepsilon.$$
Fix such $n$ and let $N=\min\{\ell>n;u^{(j)}_\ell\not\in\{u^{(i)}_1,\ldots,u^{(i)}_n\}\}$. If $N$ is finite, then $A_{N-1}=\emptyset$ and does not satisfy the requirement that $|A_N|>N\varepsilon$. It follows that $N$ is infinite, that at least one integer in $\{u^{(i)}_1,\ldots,u^{(i)}_n\}$ is repeated infinitely many times and that $\limsup_nM_n\neq\emptyset$.
\end{proof}

\section*{Acknowledgements}
This research was supported by the Australian Research Council Grant DP180100613.

RJ Williams gratefully acknowledges the support of NSF grant DMS-1712974 and thanks the Center for Modeling of Stochastic Systems at Monash University for their generous hospitality  during her sabbatical visit there.

\end{document}